\documentclass[a4paper,12pt]{article}
\usepackage[a4paper, margin=3.5cm, top=4cm]{geometry}
\usepackage{amssymb,amsmath,amsthm}
\usepackage[dvipsnames]{xcolor}
\usepackage[pdftex]{graphicx}
\usepackage{mathabx}
\usepackage{tikz} 
\usepackage{wasysym}
\usepackage[all]{xy}
\usepackage{relsize}
\usepackage{pgf}
\usepackage{hyperref}

\usetikzlibrary{positioning}
\usetikzlibrary{cd}
\usetikzlibrary{intersections}

\newtheorem{theorem}{Theorem}[section]
\newtheorem{lemma}[theorem]{Lemma}
\newtheorem{proposition}[theorem]{Proposition}
\newtheorem{corollary}[theorem]{Corollary}

\newtheorem{definition}[theorem]{Definition}
\newtheorem{proposition-definition}[theorem]{Proposition-Definition}

\theoremstyle{remark}
\newtheorem*{remark}{Remark}
\newtheorem*{example}{Example}

\newcommand{\oo}{\mathcal{O}}

\newcommand{\mc}{\mathcal}
\newcommand{\mbb}{\mathbb}

\DeclareMathOperator{\GL}{GL}

\definecolor{mypink1}{RGB}{231, 62, 156}
\definecolor{mypink2}{RGB}{219, 48, 122}
\definecolor{mygreen}{RGB}{107, 165, 51}
\definecolor{mygreen2}{RGB}{16,131,58}
\definecolor{teal}{RGB}{72, 146, 182}
\definecolor{lightpurple} {RGB}{172, 82, 204}
\definecolor{darkpurple}{RGB}{103,17,168}
\definecolor{myred1}{RGB}{128,0,32}
\definecolor{myblue1}{RGB}{24,189,225}
\definecolor{myblue2}{RGB}{16,25,144}

\begin{document}

\title{Measure and Integration on $\GL_2$ over a Two-Dimensional Local Field} \date{} \author{Raven Waller\thanks{This work was completed while the author was supported by an EPSRC Doctoral Training Grant at the University of Nottingham.}}

\maketitle

\begin{abstract}
We define a translation-invariant measure and integral on $\GL_2$ over a two-dimensional local field $F$ by combining elements of the classical $\GL_2$ theory and the theory developed by Fesenko for the field $F$ itself. We give several alternate expressions for the integral, including one which agrees with the integral defined previously by Morrow.
\end{abstract}

\section{Introduction}
For a nonarchimedean local field $L$, the topology, analysis, and arithmetic on $L$ are intimately related. Inside $L$, one has its subring $\mc{O}_L$ of integral elements. The set of all translates $\alpha + \pi^i \mc{O}_L$ of fractional ideals of $\mc{O}_L$ is basis for the topology on $L$ (here $\pi$ is a prime element of $\mc{O}_L$). With respect to this topology, $L$ is a locally compact topological field, and so we have a Haar measure $\mu_L$ on $L$, which is usually normalised so that $\mu(\mc{O}_L)=1$.

These relationships transfer quite naturally to the group $\GL_n(L)$ of $n \times n$ invertible matrices over $L$. Since $L$ is locally compact, so is $\GL_n(L)$. Its maximal compact subgroup is $\GL_n(\mc{O}_L)$, which in turn contains the compact subgroups $K_m = I_n + \pi^m M_2(\mc{O}_L)$, which are the matrix analogues of the higher unit groups of $L$.

When one moves to a higher dimensional local field $F$, these relationships begin to break down. In particular, the topology on $F$ (which is no longer locally compact) becomes much more separated from its arithmetic structures. Furthermore, the loss of local compactness means there is no real valued Haar measure on $F$. 

The problem of finding a suitable replacement for the Haar measure remained open for several decades. However, in the early 2000s, Fesenko made the remarkable observation that, if one extends the space of values from $\mbb{R}$ to $\mbb{R}(\!(X_2)\!) \cdots(\!(X_n)\!)$, one obtains a theory of Harmonic analysis on an $n$-dimensional local field $F$ which generalises the one-dimensional theory quite well.

By making the above extension, Fesenko defined in \cite{fesenko-aoas1}, \cite{fesenko-loop} a translation-invariant measure $\mu$ on the ring of subsets of $F$ generated by translates of fraction ideals of the ring $O_F$ of integers of $F$ with respect to a discrete valuation of rank $n$. Interestingly, he noted that his generalised measure may fail to be countably additive in certain cases, although this can only happen if the infinite series which arise do not absolutely converge.

Several years later a similar idea was considered by Morrow in \cite{morrow-2dint}. Rather than appealing to a measure and writing down an explicit family of measurable sets, he defined an integral on $F$ directly by lifting integrable functions from the residue field. In fact, his constructions work in the more general setting when $F$ is any split valuation field (i.e. a field equipped with a valuation $v: F^\times \rightarrow \Gamma$ such that there is a map $t : \Gamma \rightarrow F^\times$ from the value group $\Gamma$ into $F^\times$ such that $v \circ t$ is the identity on $\Gamma$) whose residue field is a local field (which may be archimedean). In the case that $F$ is an $n$-dimensional local field, Fesenko's integral is recovered.

Morrow then went on to apply his construction to the group $\GL_n(F)$ in \cite{morrow-gln}, and he considered in detail the effect of performing linear changes of variables. He then went on in \cite{morrow-fubini} to notice that when one performs changes of variables which are not linear, Fubini's Theorem may fail. He also noted that to define an integral on an arbitrary algebraic group would probably require a more general approach.

In this paper, we instead construct an $\mbb{R}(\!(X)\!)$-valued, finitely additive measure on $\GL_2(F)$, where $F$ is a two-dimensional nonarchimedean local field, using Fesenko's explicit approach. As our `generating sets' we choose analogues of the compact subgroups $K_m$, namely
\begin{displaymath}
K_{i,j} := I_2 + t_1^i t_2^j M_2(O_F),
\end{displaymath}
where $O_F$ is the rank two ring of integers of $F$, and $t_1,t_2$ are local parameters for $F$ (so $t_1$ generates the maximal ideal of $O_F$, $t_2$ generates the maximal ideal of the rank one integers $\mc{O}_F$). The groups $K_{i,j}$ fit into the following lattice.
\begin{displaymath}
\xymatrix@=1.2em{
& \vdots & \vdots & \vdots & \vdots & \\
\cdots \ar@{}[r]|-*[@]{\supset} & K_{-1,2} \ar@{}[r]|-*[@]{\supset} \ar@{}[u]|-*[@]{\supset} & K_{0,2} \ar@{}[r]|-*[@]{\supset} \ar@{}[u]|-*[@]{\supset} & K_{1,2} \ar@{}[r]|-*[@]{\supset} \ar@{}[u]|-*[@]{\supset} & K_{2,2} \ar@{}[u]|-*[@]{\supset} \ar@{}[r]|-*[@]{\supset} & \cdots \\
\cdots \ar@{}[r]|-*[@]{\supset} & K_{-1,1} \ar@{}[r]|-*[@]{\supset} \ar@{}[u]|-*[@]{\supset} & K_{0,1} \ar@{}[r]|-*[@]{\supset} \ar@{}[u]|-*[@]{\supset} & K_{1,1} \ar@{}[r]|-*[@]{\supset} \ar@{}[u]|-*[@]{\supset} & K_{2,1} \ar@{}[u]|-*[@]{\supset} \ar@{}[r]|-*[@]{\supset} & \cdots \\
&&K \ar@{}[u]|-*[@]{\supset} \ar@{}[r]|-*[@]{\supset} & K_{1,0} \ar@{}[u]|-*[@]{\supset} \ar@{}[r]|-*[@]{\supset}  & K_{2,0} \ar@{}[u]|-*[@]{\supset} \ar@{}[r]|-*[@]{\supset} & \cdots
}
\end{displaymath}

From the point of view of this theory, $\GL_2(F)$ behaves to some extent as if it is a locally compact group with maximal compact subgroup $K=\GL_2(O_F)$. In particular, it is often (though not always) more appropriate to work with objects associated to $O_F$ than to the rank one integers $\mc{O}_F$, a phenomenon which is readily apparent in Fesenko's work, as well as in several other places (such as \cite{kim-lee-spherical}, \cite{lee-iwahori}).

What is interesting to note is that, while Morrow's general approach is more abstract, the explicit approach constructed here in fact descends to many other algebraic groups quite naturally. This then gives a sound starting point for one to search for a theory in the sense of Morrow for algebraic groups: if such a theory exists, it should at the very least give the same results as the less general but more direct approach wherever the two intersect.

At first it may seem paradoxical that in order to make further generalisations one must lose generality, but this phenomenon appears to be very common in higher dimensional number theory. When one leaves behind everything which prevents the crossing of ``dimensional barriers", and then builds the higher dimensional theory on the remaining foundations, many exciting similarities seem to appear.

The contents of this paper are as follows. First of all, we review both the one-dimensional theory for $\GL_2$ and Fesenko's definition of the measure on a two-dimensional field in section \ref{sec:review}. In section \ref{sec:structure} we study the structure of $\GL_2$ over a two-dimensional local field, looking closely at the properties of the distinguished subgroups $K_{i,j}$. This section culminates in the definition of the ring $\mc{R}$ of measurable subsets, which roughly speaking is generated by the $K_{i,j}$.

In section \ref{sec:measure} we then define a left-invariant, $\mbb{R}(\!(X)\!)$-valued measure $\mu$ on $\mc{R}$, closely connected to both theories discussed in section \ref{sec:review}, such that $\mu(\GL_2(O_F))=1$ and
\begin{displaymath}
\mu(K_{i,j}) = \frac{q^3}{(q^2-1)(q-1)} q^{-4i}X^{4j}.
\end{displaymath}
We then show that this measure is well-defined in two important steps. First, we show that if we write a particular measurable set in two different ways, there exists a common ``refinement" of both presentations. Then, we show that the measure is well-defined when one passes from a measurable set to its refinement.

This is followed immediately by the definition of the integral in section \ref{sec:integration}. After some preliminaries, we establish the classical formula
\begin{displaymath}
\begin{split} \begin{flalign*} \int_{\GL_2(F)} f(g) \ d \mu(g) = && \end{flalign*} \\
\frac{q^3}{(q^2-1)(q-1)} \int_{F} \int_{F} \int_F \int_{F} \frac{1}{|\alpha \delta - \beta \gamma|_F^2} f \bigg(
\bigg( \scalebox{0.8}{$\begin{array}{cc}
\alpha & \beta    \\
\gamma & \delta
 \end{array} $} \bigg) \bigg) \ d \alpha \ d \beta \ d \gamma \ d \delta,
\end{split}
\end{displaymath}
which also coincides (up to a known constant) with the integral defined by Morrow in \cite{morrow-gln}. Using this formula, we then deduce several anticipated properties of the measure, including right-invariance and (in a certain refined sense) countable additivity. 

Finally, we include as an appendix \ref{sec:morrow} a brief discussion regarding the similarities and differences between the approach to integration considered here and the previous approach due to Morrow.

\textbf{Notation.} Throughout this paper we will use the following notation. Unless specified otherwise, $F$ will always denote a $2$-dimensional local field. $\oo_F$ and $O_F$ will be (respectively) the rank-one and rank-two integers of $F$. $E=\overline{F}$ will denote the (first) residue field of $F$ and $\mc{O}_E=\mc{O}_{\overline{F}}$ its ring of integers. Let $q$ be the number of elements in the finite field $\overline{E}$. Fix a rank two valuation $v: F^\times \rightarrow \mbb{Z} \times \mbb{Z}$, where the latter is ordered lexicographically from the right (so $(1,0) < (0,1)$), and fix a pair of local parameters $t_1,t_2$. With this notation, $t_2 \oo_F$ is the maximal ideal of $\oo_F$ and $t_1 O_F$ is the maximal ideal of $O_F$.

We denote the set of all $2 \times 2$ matrices with entries in $F$ by $M_2(F)$, and the subset of matrices invertible over $F$ by $\GL_2(F)$. We also define $M_2(O_F)$ and $\GL_2(O_F)$ similarly, and we write $I_2$ for the $2 \times 2$ identity matrix.

\textbf{Acknowledgements.} I would like to thank my supervisor Ivan Fesenko for suggesting this topic and for his support throughout the writing of this paper. I also want to thank everyone who attended the Kac-Moody Groups and $L$-Functions meeting in Nottingham in October 2016, where I gained several new ideas and insights into this work. In particular I am grateful to Kyu-Hwan Lee, Thomas Oliver, and Wester van Urk for their many helpful comments and suggestions. I would also like to thank Matthew Morrow for reading an earlier version of this text and providing several important comments. Finally, I am grateful to the anonymous referee for their comments and suggestions.

\section{Brief review of the existing theory} \label{sec:review}

In this section we give a brief outline of both the two-dimensional local field theory and the theory for $\GL_2$ over a one-dimensional field. We begin with the more classical one-dimensional theory.

Let $L$ be a nonarchimedean local field, let $\mc{O}_L$ be its ring of integers, and let $\pi \in \mc{O}_L$ be a generator of the maximal ideal. Both $L$ and $\GL_2(L)$ are locally compact groups, containing $\mc{O}_L$ and $\GL_2(\mc{O}_L)$ respectively as compact subgroups. Any maximal compact subgroup of $\GL_2(L)$ is isomorphic to $K=\GL_2(\mc{O}_L)$. For every integer $n \ge 1$, $K$ contains the compact subgroup $K_n = I_2 + \pi^n M_2(\mc{O}_L)$.

To define a Haar measure $\mu$ on $\GL_2(L)$ it is sufficient to specify the values $\mu(K_n)$ (see chapter 6 of \cite{goldfeld-hundley1} for details when $L=\mbb{Q}_p$). If we normalise so that $\mu(\GL_2(\oo_L))=1$, the only choice is $\mu(K_n) = |\GL_2(\oo_L):K_n|^{-1}$.

Using this measure, one then defines an integral on $\GL_2(L)$. In particular, for a locally constant function $f=\sum_i c_i \cdot char_{U_i}$, where $c_i \in \mbb{C}$ and $U_i \subset \GL_2(F)$ are measurable, we have $\int_{\GL_2(L)} f d \mu = \sum c_i \mu(U_i)$. 

The integral can be represented in matrix coordinates. Suppose $g = \bigg( \scalebox{0.8}{$\begin{array}{cc}
\alpha & \beta    \\
\gamma & \delta
 \end{array} $} \bigg) \in \GL_2(L)$. Then
\begin{equation*}
\int_{\GL_2(L)} f(g) \ d \mu(g) = c \int_{L} \int_{L} \int_L \int_{L} \frac{1}{|\alpha \delta - \beta \gamma|_L^2} f \bigg(
\bigg( \scalebox{0.8}{$\begin{array}{cc}
\alpha & \beta    \\
\gamma & \delta
 \end{array} $} \bigg) \bigg) \ d \alpha \ d \beta \ d \gamma \ d \delta.
\end{equation*}
Here, if $\mu_L$ denotes the measure on $L$, $|x|_L := \dfrac{\mu_L(xU)}{\mu_L(U)}$ for any measurable subset $U \subset L$ of nonzero measure, and $d \alpha = d \mu_L(\alpha)$, $d \beta = d \mu_L(\beta)$,  $d \gamma = d \mu_L(\gamma)$, $d \delta = d \mu_L(\delta)$. The constant $c = \dfrac{q^3}{(q^2-1)(q-1)}$ for the normalisation $\mu(\GL_2(\oo_L))=1$.

Now we turn to the measure defined by Fesenko in \cite{fesenko-aoas1} and \cite{fesenko-loop} for a two-dimensional local field $F$. For (a translate of) a fractional ideal $\alpha + t_1^i t_2^j O_F$, one defines $\mu(\alpha + t_1^i t_2^j O_F)=q^{-i}X^j$. This yields a finitely-additive, translation invariant measure on the ring of subsets of $F$ generated by sets of the above form, which is countably additive in a refined sense. (Since the measure on $\GL_2(F)$ will satisfy the same property we do not elaborate on this here, and instead refer to Corollary \ref{countadd}.) 

Instead of taking values in $\mbb{R}$, this measure instead takes values in the two-dimensional archimedean local field $\mbb{R}(\!(X)\!)$. From a topological perspective, the element $X$ should be smaller than every positive real number but greater than zero, and so it may be interpreted as an infinitesimal positive element.

One then proceeds to define a $\mbb{C}(\!(X)\!)$-valued integral on functions of the form $f=\sum_i c_i \cdot char_{U_i}$ in the same way as one does for locally constant functions on $\GL_2$ of a local field above. The integral of a function which is zero away from finitely many points is also defined to be $0$. Fesenko also extends the class of integrable functions to include characters of $F$, but this will not be of importance to us in the current paper.

We end this section with an important definition. For $\alpha \in F^\times$, let $|\alpha|_F = \dfrac{\mu(\alpha U)}{\mu(U)}$, where $U$ is any measurable subset of $F$ with nonzero measure. This does not depend on the choice of $U$. We also put $|0|_F=0$. This is an extension of the notion of absolute value $|\cdot|_L$ on a local field $L$, since the definition we have just given is equivalent to the usual definition in this case. However, since $|\cdot|_F$ takes values in $\mbb{R}(\!(X)\!)$ rather than $\mbb{R}$, we do not use the term `absolute value' for this function. (In \cite{fesenko-aoas1}, \cite{fesenko-loop} this is called the module.) If $\varepsilon \in O_F^\times$, we have $|t_1^i t_2^j \varepsilon|_F = q^{-i}X^j$.

\section{The structure of $\GL_2(F)$ and its subgroups}  \label{sec:structure}

We now aim to emulate the one-dimensional results for $\GL_2$ in dimension two, following the blueprints of Fesenko's measure via distinguished sets. From now on we will use the notation as defined in the first section. 

We begin with the two-dimensional analogue of the compact subgroups $K_n$.

\begin{definition}
For $(i,j) > (0,0)$ in $\mbb{Z} \oplus \mbb{Z}$, put $K_{i,j} = I_2 + t_1^i t_2^j M_2 (O_F)$. One easily checks that $K_{i,j} \subset K_{m,n}$ if $(m,n) \le (i,j)$. It will be convenient to set $K=\GL_2(O_F)$.
\end{definition}

\begin{lemma} \label{Kijnormal}
$K_{i,j}$ is a normal subgroup of $K$.
\end{lemma}

\begin{proof}
Let  \[ g = \bigg( \scalebox{0.8}{$\begin{array}{cc}
a & b    \\
c & d  
 \end{array} $} \bigg) \in K, \ 
k= \bigg( \scalebox{0.8}{$\begin{array}{cc}
1+t_1^i t_2^j \alpha &  t_1^i t_2^j \beta   \\
t_1^i t_2^j \gamma & 1 + t_1^i t_2^j \delta  
 \end{array} $} \bigg) \in K_{i,j}. \] 
Then
\[ gkg^{-1} = \bigg( \scalebox{0.8}{$\begin{array}{cc}
1 + t_1^i t_2^j w & t_1^i t_2^j x    \\
t_1^i t_2^j y & 1 + t_1^i t_2^j z
 \end{array} $} \bigg) \in K_{i,j}, \]
since $w = \dfrac{\alpha ad + \gamma bd - \beta a c - \delta b c}{ad-bc}$, $x = \dfrac{\beta a^2 + (\delta - \alpha) ab - \gamma b^2}{ad-bc}$, \linebreak $y = \dfrac{\gamma d^2 + (\alpha - \delta) cd - \beta c^2}{ad-bc}$, $z = \dfrac{\beta ac + \delta ad - \alpha bc - \gamma bd}{ad-bc}$ are all elements of $O_F$ (note that $ad-bc = \det g \in O_F^\times$ by assumption).
\end{proof}

With a view towards defining an invariant measure on $\GL_2(F)$, we study further important properties of the subgroups $K_{i,j}$.

\begin{lemma}  \label{intproperty}
Let $(i,j) \le (m,n)$, and let $g, h \in \GL_2(F)$. Then the intersection of $g K_{i,j}$ and $h K_{m,n}$ is either empty or equal to one of them.
\end{lemma}

\begin{proof}
By assumption $K_{m,n} \subset K_{i,j}$, hence $g K_{i,j} \cap h K_{m,n} \subset g K_{i,j} \cap h K_{i,j}$. The latter two sets are $\GL_2(F)$-cosets of the same subgroup $K_{i,j}$, hence are disjoint or equal. If they are disjoint then $g K_{i,j}$ and $h K_{m,n}$ are also disjoint. If they are equal then $g K_{i,j} \cap h K_{m,n} = h K_{i,j} \cap h K_{m,n} = h K_{m,n}$.
\end{proof}

Note that, by this intersection property, any union $g K_{i,j} \cup h K_{m,n}$ is either disjoint or equal to one of the components. 

\begin{remark}
Although Lemma \ref{intproperty} is very simple, it will be used more frequently than any other result in this paper, and so we would like to draw attention to its importance here.
\end{remark}

Following section 6 of \cite{fesenko-loop}, we make the following definitions.

\begin{definition}
A distinguished set is either empty or a set of the form $g K_{i,j}$ with $g \in \GL_2(F)$ and $(i,j)>(0,0)$. A dd-set is a set of the form $A= \bigcup_i A_i \backslash (\bigcup_j B_j)$, with pairwise disjoint distinguished sets $A_i$, pairwise disjoint distinguished sets $B_j$, and $\bigcup_j B_j \subset \bigcup_i A_i$. A ddd-set is a disjoint union of dd-sets.
\end{definition}

\begin{lemma}
The class $\mc{R}$ of ddd-sets is closed under union, intersection, and difference. $\mc{R}$ is thus the minimal ring (of sets) which contains all of the distinguished sets.
\end{lemma}

\begin{proof}
First we show that the class of dd-sets is closed under intersection. Since intersection distributes over union, it is enough to check this for two dd-sets of the form $E_1  = A \backslash \bigcup_i B_i$ and $E_2 = C \backslash \bigcup_j D_j$ with $A, B_i, C,$ and $D_j$ all distinguished. We have $$E_1 \cap E_2 = (A \cap C) \backslash \left( \bigcup_i B_i \cup \bigcup_j D_j \right),$$ and since $A$ and $C$ are distinguished their intersection is distinguished also by Lemma \ref{intproperty}. $E_1 \cap E_2$ is thus a dd-set by definition.

Since the classes of distinguished sets and dd-sets are not closed under unions, if we show that the class $\mc{R}$ of ddd-sets is closed under union, intersection, and difference then by construction it is the minimal ring containing the distinguished sets. However, $\mc{R}$ is closed under unions by definition, and is closed under intersections by the same argument in the previous paragraph, and so it remains to prove that it is closed under differences.

As before let $E_1  = A \backslash \bigcup_i B_i$ and $E_2 = C \backslash \bigcup_j D_j$ with $A, B_i, C, D_j$ all distinguished. Using de Morgan's laws one obtains $$E_1 \backslash E_2 = \left(A \backslash \left( C \cup \bigcup_i B_i \right) \right) \cup \left( \bigcup_j (A \cap D_j) \backslash \bigcup_i B_i\right).$$ Both components are ddd-sets by definition, hence the union is a ddd-set as required.
\end{proof}

In Section \ref{sec:measure} we will show that there exists a finitely additive, translation-invariant measure on $\mc{R}$ taking values in $\mbb{R} (\!(X)\!)$. To this end, we outline some useful properties of ddd-sets.

\begin{definition}
For two distinguished sets $gK_{i,j} \supset hK_{m,n}$, by abuse of language we define the index $|gK_{i,j} : h K_{m,n}|$ of $hK_{m,n}$ in $gK_{i,j}$ to be the index $|K_{i,j}: K_{m,n}|$. 
\end{definition}

\begin{remark}
This definition makes sense since $g^{-1} h K_{m,n}$ is a coset of $K_{m,n}$ inside $K_{i,j}$.
\end{remark}

\begin{lemma} \label{indexlemma}
If $j=n$, $|gK_{i,j} : h K_{m,n}|=q^{4(m-i)}$. Otherwise, the index is infinite.
\end{lemma}

\begin{proof}
The map $K_{i,j} \rightarrow \left( t_1^i t_2^j O_F / t_1^{i+1} t_2^j O_F \right)^4 \simeq \left( O_F / t_1 O_F \right)^4$ given by $$\bigg( \scalebox{0.8}{$\begin{array}{cc}
1+t_1^i t_2^j \alpha &  t_1^i t_2^j \beta   \\
t_1^i t_2^j \gamma & 1 + t_1^i t_2^j \delta  
 \end{array} $} \bigg) \mapsto (\alpha, \beta, \gamma, \delta) \mod t_1 O_F$$ induces an isomorphism $K_{i,j} / K_{i+1,j} \simeq (O_F / t_1 O_F)^4$. We thus have $|gK_{i,j} : h K_{m,j}| = |K_{i,j}:K_{m,j}| = \prod_{r=i}^{m-1} |K_{r,j}:K_{r+1,j}| = |O_F / t_1 O_F|^{4(m-i)} = q^{4(m-i)}$.

On the other hand, the map $K_{i,j} \rightarrow \left( t_1^i t_2^j O_F / t_1^i t_2^{j+1} O_F \right)^4 \simeq \left( O_F / t_2 O_F \right)^4$ given by $\bigg( \scalebox{0.8}{$\begin{array}{cc}
1+t_1^i t_2^j \alpha &  t_1^i t_2^j \beta   \\
t_1^i t_2^j \gamma & 1 + t_1^i t_2^j \delta  
 \end{array} $} \bigg) \mapsto (\alpha, \beta, \gamma, \delta) \mod t_2 O_F$ induces an isomorphism $K_{i,j} / K_{i,j+1} \simeq (O_F / t_2 O_F)^4$, and the latter group has infinite order.
\end{proof}

\begin{lemma} \label{indexink}
The index $|K:K_{i,0}|= q^{4i-3} (q^2-1)(q-1)$, and for $j>0$ the index $|K:K_{i,j}|$ is infinite for any $i$.
\end{lemma}

\begin{proof}
The map $K \rightarrow \GL_2 (O_F / t_1 O_F)$, $g \mapsto g \mod t_1 O_F$ induces an isomorphism $K/K_{1,0} \simeq \GL_2(O_F / t_1 O_F) \simeq \GL_2 ( \mbb{F}_q )$, and the latter is well known to have order $(q^2-1)(q^2-q)$. We thus have $|K:K_{i,0}| = |K:K_{1,0}| \cdot |K_{1,0}:K_{i,0}| = q^{4(i-1)} (q^2-1)(q^2-q) = q^{4i-3} (q^2-1)(q-1)$.

Now suppose $j>0$. For $i \ge 0$, $|K:K_{i,j}| \ge |K_{0,j} : K_{i,j}|$, and the latter is infinite by Lemma \ref{indexlemma}. On the other hand, if $i<0$, the index $|K:K_{i,j}|$ differs from $|K_{0,j} : K_{i,j}|$ only by a finite constant, and so in either case $|K:K_{i,j}|$ is infinite.
\end{proof}

\begin{proposition} \label{allfiniteindex}
If $D$ is a distinguished set such that $D = \bigcup_{r=1}^n D_r$ is a disjoint union of finitely many distinguished sets $D_r$ then the indices $|D:D_r|$ are all finite.
\end{proposition}

\begin{proof}
By making a translation if necessary, we may assume that $D=K_{i,j}$ for some $(i,j)$. The first step is to show that at least one $D_r$ has finite index in $D$. If $D = \bigcup_{r=1}^n D_r = \bigcup_{r=1}^n g_r K_{i_r, j_r}$, let $(i_*, j_*) = \min_r \{ (i_r, j_r) \}$. Then $D \supset \bigcup_r g_r K_{i_*, j_*} \supset \bigcup_r g_r K_{i_r, j_r} = D$, and so $K_{i_*, j_*}$ has finite index in $D$.

By relabelling if necessary, we may thus assume that $|D:D_1|$ is finite, and so we may take a complete system of coset representatives $S=\{h_1=I_2, h_2, \dots, h_m \}$. We thus have $D \backslash D_1 = \bigcup_{r=2}^n D_r = \bigcup_{s=2}^m h_s D_1$. Taking the intersection with any $D_r$, this gives $D_r = \bigcup_{s=2}^m \left( D_r \cap h_sD_1 \right)$.

For each $r>1$, let $S_r = \{ h \in S : D_r \cap hD_1 \neq \emptyset \}$. Each $S_r$ is nonempty, since $\bigcup_{h \in S_r} (D_r \cap hD_1) = D_r$. By Lemma \ref{intproperty}, $D_r \cap hD_1$ is thus equal to either $D_r$ or $hD_1$ for any $h \in S_r$. 

If $D_r \cap hD_1 = D_r$ for any $h$ then we must have $S_r = \{ h \}$ and $D_r = hD_1$, hence $D_r$ has finite index in $D$ as a translate of $D_1$. On the other hand, if $D_r \cap hD_1 = hD_1$ for all $h \in S_r$ then we have $h D_1 \subset D_r$. By the tower law for indices we thus have $|D:D_1| = |D:hD_1| = |D:D_r| \cdot |D_r : hD_1|$, hence $|D:D_r| \le |D:D_1|$ is finite.
\end{proof}

\begin{corollary}
If $D$, $D_1, \dots, D_n$ are finitely many distinguished sets such that $D_r \subset D$ and $|D:D_r|$ is infinite for each $r$, there do not exist finitely many distinguished sets $C_1, \dots, C_m$ with $D \backslash \bigcup_{r=1}^n D_r = \bigcup_{s=1}^m C_s$.
\end{corollary}

\begin{proof}
We may assume that the $D_r$ are all disjoint by deleting any which are contained in some larger one. Similarly, we may assume that the $C_s$ are all disjoint. If there did exist finitely many such $C_s$, we would thus have $D = \bigcup_{r=1}^n D_r \cup \bigcup_{s=1}^m C_s$, a union of finitely many disjoint distinguished sets. By Proposition \ref{allfiniteindex} this in particular implies that each $|D:D_r|$ is finite, which contradicts our assumption that these indices are infinite.
\end{proof}

In order to show that the measure that we will construct in the next section is well-defined, it is useful to introduce the idea of a refinement of a ddd-set.

\begin{definition}
Let $A = \bigcup_i  B_i $ be a ddd-set, where the $$B_i = \bigcup_j C_{i,j} \backslash \bigcup_k D_{i,k}$$ are disjoint dd-sets made from distinguished sets $C_{i,j}$ and $D_{i,k}$. A refinement of $A$ is a ddd-set $$\tilde A = \bigcup_p \left( \bigcup_q X_{p,q} \backslash \bigcup_r Y_{p,r} \right)$$ satisfying the following conditions:
\begin{enumerate}
\item $A = \tilde A$ as sets;
\item For every $(i,j)$ there is some $(p,q)$ such that $X_{p,q} = C_{i,j}$;
\item For every $(i,k)$ there is some $(p,r)$ such that $Y_{p,r} = D_{i,k}$.
\end{enumerate}
\end{definition}

The idea of a refinement can be best understood from the following picture.

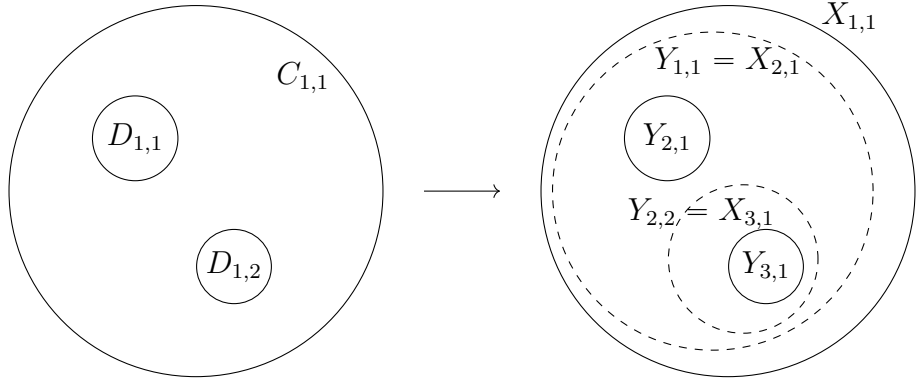
\begin{displaymath}
\begin{tikzpicture}
\draw (-3.5,0) circle (70pt); 
\draw (-4.3,0.7) circle (16pt);
\draw (-3, -1) circle (14pt);
\draw (-1.6,1.5) node[left] {$C_{1,1}$};
\draw (-4.3,0.7) node {$D_{1,1}$};
\draw (-3, -1) node {$D_{1,2}$};

\draw (3.5,0) circle (70pt); 
\draw (2.7,0.7) circle (16pt);
\draw (4, -1) circle (14pt);
\draw (5.1,2.3) node {$X_{1,1}$};
\draw (2.7, 0.7) node {$Y_{2,1}$};
\draw (4,-1) node {$Y_{3,1}$};

\draw [dashed] (3.3,0) circle (60pt); 
\draw [dashed] (3.7,-0.9) circle (28pt);
\draw (3.5, 1.7) node {$Y_{1,1} = X_{2,1}$};
\draw (3.15,-0.3) node {$Y_{2,2} = X_{3,1}$};

\draw [->] (-0.5,0) to (0.5,0);

\end{tikzpicture}
\end{displaymath}

Here, the ddd-set on the left is $A = C_{1,1} \backslash \left( D_{1,1} \cup D_{1,2} \right)$. The set on the right $\tilde A = \left( X_{1,1} \backslash Y_{1,1} \right) \cup \left( X_{2,1} \backslash \left( Y_{2,1} \cup Y_{2,2} \right) \right) \cup \left( X_{3,1} \backslash Y_{3,1} \right)$ is a refinement of $A$ since $C_{1,1} = X_{1,1}$, $D_{1,1} = Y_{2,1}$, and $D_{1,2} = Y_{3,1}$.

\begin{remark}
The idea of refinements of ddd-sets comes from the use of refinements of open intervals in the construction of the Riemann integral.
\end{remark}

Before coming to the fundamental result regarding refinements, it is useful to introduce some terminology for ddd-sets.

\begin{definition}
Let $$A=\bigcup_i \left( \bigcup_j C_{i,j} \backslash \bigcup_k D_{i,k} \right)$$ be a ddd-set. We call $A$ reduced if it does not contain any dd-components of the form $B \backslash B$ for a distinguished set $B$.
\end{definition}

\begin{remark}
The property of being reduced depends on the particular components of a ddd-set (or, more precisely, depends on the specific \textit{presentation} of a given ddd-set). For example, if $B, C, D$ are disjoint nonempty distinguished sets such that $A = B \backslash (C \cup D)$, $A$ is reduced even if $C \cup D = B$. However, $A' = \left( B \backslash (C \cup D) \right) \cup \left( E \backslash E \right)$ is not reduced, even though $A=A'$ at the level of sets.
\end{remark}

Note that for any ddd-set $A$ we may form a reduced ddd-set $A_{red}$ by removing all of the superfluous components $B \backslash B$. Since the components we delete are empty, $A = A_{red}$ at the level of sets.

\begin{definition}
Let $$A=\bigcup_i \left( \bigcup_j C_{i,j} \backslash \bigcup_k D_{i,k} \right)$$ be a ddd-set. The components $C_{i,j}$ are called the big shells, and the components $D_{i,k}$ are called the small shells.
\end{definition}

\begin{remark}
By the above definition, we can reformulate the definition of a refinement as follows. A refinement of a ddd-set $A$ is a ddd-set $\tilde A$ such that $A = \tilde A$ as sets, every big shell of $A$ is a big shell of $\tilde A$, and every small shell of $A$ is a small shell of $\tilde A$.
\end{remark}

The most fundamental result regarding refinements is as follows.

\begin{theorem} \label{commonrefinement}
Let $A$ and $A'$ be reduced ddd-sets with $A=A'$ as sets. Then there exists a reduced ddd-set $\tilde A$ which is a refinement of both $A$ and $A'$.
\end{theorem}

\begin{proof}
Suppose $$A=\bigcup_i \left( \bigcup_j C_{i,j} \backslash \bigcup_k D_{i,k} \right),$$ $$A' = \bigcup_\ell \left( \bigcup_m C'_{\ell,m} \backslash \bigcup_n D'_{\ell,n} \right),$$ for distinguished sets $C_{i,j}, D_{i,k}, C'_{\ell,m}, D'_{\ell,n}$. Starting with $S=A$, the following algorithm will give the required $\tilde A$.

\textbf{Step 1}: The set $S$ is a reduced ddd-set which is a refinement of $A$ by construction. If $S$ is a refinement of $A'$ then we may take $\tilde A = S$ and we are done. If there is some $C'_{\ell,m}$ which is not a big shell of $S$, go to Step 2. Otherwise, there is some $D'_{\ell,n}$ which is not a small cell of $S$, in which case go to step 4.

\textbf{Step 2}: If $C'_{\ell,m} \not \subset C$ for all big shells $C$ of $S$ then go to Step 3. Otherwise, the set of all big shells of $S$ containing $C'_{\ell,m}$ is nonempty and totally ordered by inclusion by Lemma \ref{intproperty} (since their intersection in particular contains $C'_{\ell,m}$), and so there is a minimal one $C_{\min}$.

Let $S'$ be the same as $S$ but with $C_{\min} \backslash \bigcup_{D \subset C_{\min}} D$ replaced with $$\left( C_{\min} \backslash \left( C'_{\ell,m} \cup \bigcup_{D \cap C'_{\ell,m} = \emptyset} D \right) \right) \cup \left( C'_{\ell,m} \backslash \bigcup_{D \subset C'_{\ell,m}} D \right).$$
(Note that none of these small shells can contain $C'_{\ell,m}$, since otherwise there would be a big shell smaller than $C_{\min}$ containing $C'_{\ell,m}$.) Since every big shell of $S$ is still a big shell in $S'$, and every small shell of $S$ is still a small shell in $S'$, $S'$ is a refinement of $S$, and hence of $A$. Furthermore, $S'$ contains one more big shell of $A'$ than $S$. Thus if we put $S=S'$ and return to Step 1, after finitely many iterations there will be no big shells of $A'$ which are not contained in at least one big shell of $S$.

\textbf{Step 3}: Since $C'_{\ell,m} \not \subset C$ for all big shells $C$ of $S$, either $C'_{\ell,m} \cap C = \emptyset$ for all such $C$, or by Lemma \ref{intproperty} there is at least one $C$ with $C \subset C'_{\ell,m}$.

In the first case, since $S=A'$ as sets, we must have $C'_{\ell,m} \backslash \bigcup_{n} D'_{\ell,n} = \emptyset$. We thus define $S'$ to be the disjoint union $S \cup \left( C'_{\ell,m} \backslash \bigcup_{n} D'_{\ell,n} \right)$. $S'$ is a refinement of $A$, and again contains one more big shell of $A'$ then $S$ does, and so returning to Step 1 with $S=S'$ will eliminate this case after finitely many iterations.

In the second case, we can take a collection $\{ C_x \}$ of maximal big shells of $S$ contained in $C'_{\ell,m}$; in other words, $C_x \subset C'_{\ell,m}$ for all $x$, and for every big cell of $S$ satisfying $C \subset C'_{\ell,m}$ there is exactly one $x$ with $C \subset C_x$. (The fact that we may have such a collection, and that the collection will be nonempty, is guaranteed by Lemma \ref{intproperty}.) 

Suppose that $C'_{\ell,m} = \bigcup_x C_x$. In this case, we can let $S'$ be the same as $S$ but with the component $\bigcup_{x} \left( C_x \backslash \bigcup_{D \subset C_x} D \right)$ replaced by $\left( C'_{\ell,m} \backslash \bigcup_x C_x \right) \cup \bigcup_{x} \left( C_x \backslash \bigcup_{D \subset C_x} D \right)$. This is again a refinement of $A$ and contains one more big shell of $A'$.

On the other hand, suppose that $C'_{\ell,m} \neq \bigcup_x C_x$. Then there can't be another component of $A$ to make up the difference, since this would give another big shell with nonempty intersection, so by assumption would have to be contained already in a $C_x$. This means there must be some small shells of $A'$ which cut out the remaining part. In other words, $C'_{\ell,m} = \bigcup_x C_x \cup \bigcup_{D'_{\ell,n} \subset C'_{\ell,m}} D'_{\ell,n}$. Since this is a union of distinguished sets, pairwise intersections are either empty or equal to one of the components, in which case we can discard the superfluous components until we have a disjoint union $C'_{\ell,m} = \bigcup_x C_x \cup \bigcup_{y} D_y$.

We then let $S'$ be the same as $S$ but with the component $$\bigcup_{x} \left( C_x \backslash \bigcup_{D \subset C_x} D \right)$$ replaced by $$\left( C'_{\ell,m} \backslash \left( \bigcup_x C_x \cup \bigcup_y D_y \right) \right) \cup \bigcup_{x} \left( C_x \backslash \bigcup_{D \subset C_x} D \right).$$ This gives a refinement of $A$ and contains one more big shell of $A'$, and since this also exhausts all possible big shell cases, returning to Step 1 with $S=S'$ will after finitely many iterations lead to all big shells of $A'$ being contained in $S$.

\textbf{Step 4}: Since $S$ contains all big shells of both $A$ and $A'$, $D'_{\ell,n}$ must be contained in a big shell of $S$. Since by Lemma \ref{intproperty} the set of all big shells containing $D'_{\ell,n}$ is totally ordered by inclusion, there is a minimal one $C_{\min}$. We then let $S'$ be the same as $S$ but with $C_{\min} \backslash \bigcup_{D \subset C_{\min}} D$ replaced with $$ \left( C_{\min} \backslash \left( D'_{\ell,n} \cup \bigcup_{D \cap D'_{\ell,n} = \emptyset} D \right) \right) \cup \left( D'_{\ell,n} \backslash \bigcup_{D \subset D'_{\ell,n}} D \right). $$ (Note that as in Step 2 none of the small shells can contain $D'_{\ell,n}$ since otherwise there would be a big shell smaller than $C_{\min}$ containing $D'_{\ell,n}$.) Then $S'$ is a refinement of $S$, and hence of $A$, and contains one more small shell of $A'$ than $S$. Thus by returning to Step 1 with $S=S'$, after finitely many iterations all small shells of $A'$ will be included in $S$. Furthermore, since this process does not remove any big shells, at this stage $S$ will be a refinement of $A'$, and so we can set $\tilde A = S$.
\end{proof}

\begin{example}
Let $D_1, \dots, D_n$ be finitely many disjoint distinguished sets such that $D= \bigcup_i D_i$ is distinguished. Taking $A=D$, $A'= \bigcup_i D_i$, the algorithm gives the chain of refinements $S_0 = D$, $S_1 = (D \backslash D_1) \cup D_1$, $S_2 = (D \backslash (D_1 \cup D_2)\!) \cup D_1 \cup D_2$, and so on, until finally after $n$ iterations we obtain $\tilde A = \left( D \backslash \bigcup_i D_i \right) \cup \bigcup_i D_i$. On the other hand, if we instead run the algorithm with $A= \bigcup_i D_i$, $A'=D$, we obtain $\tilde A = \left( D \backslash \bigcup_i D_i \right) \cup \bigcup_i D_i$ after a single iteration.
\end{example}

\begin{remark}
One of the most important uses of the algorithm is when we take $A'$ to be a refinement of $A$. In this case, the output $\tilde A$ of the algorithm will be $A'$, and so it gives a precise construction of $A'$ from $A$.
\end{remark}

\section{Measure on $\GL_2(F)$}  \label{sec:measure}
We shall now utilise the results of the previous section to define an invariant measure $\mu$ on the ring $\mc{R}$ of subsets of $\GL_2(F)$, and show that this measure is well-defined. We begin with the value of $\mu$ on distinguished sets.

To begin with, we would like the measure we define to be left-invariant, and so we insist that $\mu(gK_{i,j}) = \mu(K_{i,j})$ for all $g \in \GL_2(F)$ and all $(i,j) > (0,0)$. Next, we recall that $K_{i,0}$ is of finite index in $K$ for every $i\ge0$ by Lemma \ref{indexink}. By writing $K$ as a finite disjoint union of cosets, and using left-invariance and finite additivity, we obtain $\mu(K) = |K : K_{i,0}| \mu (K_{i,0})$. If we normalise so that $\mu(K)=1$, this gives $$\mu (K_{i,0}) = \frac{1}{|K : K_{i,0}|} = \frac{q^3}{(q^2-1)(q-1)} q^{-4i}.$$

If $j>0$ then the index is no longer finite. In this case, if also $i>0$ we follow Fesenko's method and define $$\mu(K_{i,j}) = \frac{X^{4j}}{|K : K_{i,0}|} = \frac{q^3}{(q^2-1)(q-1)} q^{-4i}X^{4j}.$$

\begin{remark}
In terms of this particular definition, the use of the indeterminate $X^4$ in the above definition may appear quite arbitrary. Indeed, the definition will still work if we replace $X^4$ with any indeterminate $Y$. However, we will see in Theorem \ref{intfactor} that the particular choice $Y=X^4$ ensures compatibility with Fesenko's measure on $F$ - in other words, the $X$ above is exactly the same $X$ which appears in Fesenko's measure.
\end{remark}

To see what should be the measure for nonpositive $i$, note that for any fixed $j>0$ the index $|K_{i,j}:K_{\ell, j}| = q^{4(\ell-i)}$ for $i \le \ell$ by Lemma \ref{indexlemma}. Finite additivity thus forces $\mu(K_{i,j}) = q^{4(\ell - i)} \mu(K_{\ell, j})$. If $i > 0$, this is in agreement with the definition given in the paragraph above. If $i<0$, setting $\ell=-i$ gives $$\mu(K_{i, j}) = q^{-8i} \mu(K_{-i,j}) = \frac{q^{-8i}X^{4j}}{|K:K_{-i,0}|} = \frac{q^3}{(q^2-1)(q-1)} q^{-4i}X^{4j}.$$ In particular, the formula is the same as for positive $i$. (Note that it would not be reasonable to write something like $|K:K_{i,0}|$ in this case, since $K_{i,0}$ is not even a group for $i<0$.)

To determine what should be the value of $\mu(K_{0,j})$ for $j>0$, we instead apply the argument of the previous paragraph with $i=0, \ell=1$ to obtain $$\mu(K_{0,j}) = q^4 \mu(K_{1,j}) = \frac{q^4 X^{4j}}{|K:K_{1,0}|} = \frac{q^3}{(q^2-1)(q-1)} X^{4j} .$$

We have thus proved the following.

\begin{proposition}  \label{ifmeasure}
Any left-invariant, finitely additive measure $\mu$ on $\mc{R}$ which takes values in $\mbb{R} (\!(X^4)\!)$ must satisfy $\mu (gK_{i,j} ) = \lambda q^{-4i}X^{4j}$ for some $\lambda \in \mbb{R}^\times$. In particular, if one normalises so that $\mu(K)=1$ then we have $$\lambda =\frac{q^3}{(q^2-1)(q-1)}.$$
\end{proposition}

\begin{remark}
Note that we have not yet shown that the map $\mu (gK_{i,j} ) = \lambda q^{-4i}X^{4j}$ extends to a measure on $\mc{R}$ - the above Proposition merely states that if there exists a left-invariant, finitely additive measure on $\mc{R}$ then it must be of this form when restricted to distinguished sets.
\end{remark}

We thus want to extend $\mu$ to dd-sets via $$\mu \left( \bigcup_i A_i \backslash \bigcup_j B_j \right) = \sum_i \mu(A_i) - \sum_j \mu(B_j)$$ and then to ddd-sets via  $\mu(C \cup D) = \mu(C) + \mu(D)$ for disjoint dd-sets $C$ and $D$. Under this extension, it is possible that $\mu$ depends heavily on the given presentation of a particular ddd-set, and so our task is now to show that if we have two different presentations $A$ and $A'$ of the same ddd-set then we in fact have $\mu(A) = \mu(A')$. We will do this using refinements.

\begin{proposition} \label{refinemeasure}
Let $D$ be a ddd-set, and let $\tilde D$ be a refinement of $D$. For any $\lambda \in \mbb{R}^\times$, the map $\mu$ which satisfies $\mu(gK_{i,j}) = \lambda q^{-4i}X^{4j}$ on distinguished sets, when extended to $\mc{R}$ by additivity, satisfies $\mu(D) = \mu(\tilde D)$. 
\end{proposition}

\begin{proof}
Using the algorithm of Theorem \ref{commonrefinement} with $A=D$ and $A'=\tilde D$, the resulting output will be $\tilde A = \tilde D$. In other words, the only operations that may occur to pass from a ddd-set to its refinement are exactly those in the proof of Theorem \ref{commonrefinement}, and so we only need to check that the value of $\mu$ is preserved by these operations.

In Step 2, the single component $C_{\min} \backslash \bigcup_{D \subset C_{\min}} D$ is replaced with $$\left( C_{\min} \backslash \left( C'_{\ell,m} \cup \bigcup_{D \cap C'_{\ell,m} = \emptyset} D \right) \right) \cup \left( C'_{\ell,m} \backslash \bigcup_{D \subset C'_{\ell,m}} D \right).$$ The former has measure $$\mu \left( C_{\min} \backslash \bigcup_{D \subset C_{\min}} D \right) = \mu(C_{\min}) - \sum_{D \subset C_{\min}} \mu(D).$$ The measure of the latter is given by $$\left( \mu(C_{\min}) - \left( \mu( C'_{\ell,m}) + \sum_{D \cap C'_{\ell,m} = \emptyset} \mu(D) \right) \right) + \left( \mu(C'_{\ell,m}) - \sum_{D \subset C'_{\ell,m}} \mu(D) \right)$$ $$=  \mu(C_{\min}) - \left( \left( \sum_{D \cap C'_{\ell,m} = \emptyset} \mu(D) \right)  + \left( \sum_{D \subset C'_{\ell,m}} \mu(D) \right) \right).$$ This agrees with the former, since by the assumption of Step 2 each of the small shells $D$ is either contained inside $C'_{\ell,m}$ or is disjoint from it, and so $$\sum_{D \subset C_{\min}} \mu(D) = \left( \sum_{D \cap C'_{\ell,m} = \emptyset} \mu(D) \right)  + \left( \sum_{D \subset C'_{\ell,m}} \mu(D) \right).$$

In the first case of Step 3, we add the single component $C'_{\ell,m} \backslash \bigcup_n D'_{\ell,n} = \emptyset$. In other words, we have to show that $\mu$ is well defined for a distinguished set $D$ which is a disjoint union $D = \bigcup_i D_i$ of finitely many distinguished sets. However, in the Example following Theorem \ref{commonrefinement} we saw that we may start with $D$ to obtain a chain of refinements $S_0 = D$, $S_1 = (D \backslash D_1) \cup D_1)$, \dots, $S_n = (D \backslash \bigcup_i D_i) \cup \bigcup_i D_i$. Moreover, this chain of refinements is constructed entirely using Step 2 of the algorithm, for which we have already proved that the value of $\mu$ is preserved. We thus have $\mu(D) = \mu(S_0) = \mu(S_n) = \mu(D \backslash \bigcup_i D_i) + \sum_i \mu(D_i)$ as required.

In the second case of Step 3, we instead add one of the the single components $$\left( C'_{\ell,m} \backslash \bigcup_x C_x \right) = \emptyset, \ \ \ \ \ \left( C'_{\ell,m} \backslash \bigcup_x C_x \cup \bigcup_y D_y \right) = \emptyset,$$ which is again the case of a distinguished union of disjoint distinguished sets. The argument of the previous paragraph then applies.

Finally, in Step 4, $C_{\min} \backslash \bigcup_{D \subset C_{\min}} D$ is replaced with $$ \left( C_{\min} \backslash \left( D'_{\ell,n} \cup \bigcup_{D \cap D'_{\ell,n} = \emptyset} D \right) \right) \cup \left( D'_{\ell,n} \backslash \bigcup_{D \subset D'_{\ell,n}} D \right). $$
The former has measure $\mu(C_{\min}) - \sum_{D \subset C_{\min}} \mu(D)$. The latter has measure $$ \left( \mu(C_{\min}) - \left( \mu(D'_{\ell,n}) + \sum_{D \cap D'_{\ell,n} = \emptyset} \mu(D) \right) \right) + \left( \mu( D'_{\ell,n}) - \sum_{D \subset D'_{\ell,n}} \mu(D) \right) $$
$$ = \mu(C_{\min}) - \left( \left( \sum_{D \cap D'_{\ell,n} = \emptyset} \mu(D)  \right) + \left(  \sum_{D \subset D'_{\ell,n}} \mu(D) \right) \right). $$ This is equal to the former since if any $D$ contains $D'_{\ell,n}$ this would contradict the minimality of $C_{\min}$, and so we have $$ \sum_{D \subset C_{\min}} \mu(D) = \left( \left( \sum_{D \cap D'_{\ell,n} = \emptyset} \mu(D)  \right) + \left(  \sum_{D \subset D'_{\ell,n}} \mu(D) \right) \right). $$
\end{proof}

\begin{theorem}  \label{GL2measure}
The map $\mu : \mc{R} \rightarrow \mbb{R}(\!(X)\!)$ given by $$\mu (gK_{i,j} ) = \frac{q^3}{(q^2-1)(q-1)} q^{-4i}X^{4j}$$ on distinguished sets and extended to $\mc{R}$ by finite additivity is a well-defined, finitely additive, left invariant measure satisfying $\mu(K) = 1$.
\end{theorem}

\begin{proof}
Translation invariance on the left and finite additivity follow immediately from the definition. Proposition \ref{refinemeasure} gives that $\mu(A) = \mu ( \tilde A)$ for any refinement $\tilde A$ of $A$, and Theorem \ref{commonrefinement} says that for any pair $A, A' \in \mc{R}$ with $A=A'$ as sets there is some $\tilde A$ which is a refinement of both, hence we have $\mu(A) = \mu(\tilde A) = \mu(A')$ - in other words the extension of $\mu$ to $\mc{R}$ is well-defined. Finally, Proposition \ref{ifmeasure} gives the required volume of $K$.
\end{proof}

\begin{remark}
Let $\bar \mu$ be the unique Haar measure on $\GL_2(\overline{F})$ satisfying $\bar \mu (\oo_{\overline{F}})=1$, and let $p:\GL_2(\oo_F) \rightarrow \GL_2(\overline{F})$ be the projection induced by the residue map. Let $K_n = I_2 + \bar t_1^n M_2(\oo_{\overline{F}})$ as in Section \ref{sec:review}. Then $\mu(I_2 + t_2^j p^{-1} (K_i - I_2)\!) = X^{4j} \bar \mu (K_i)$. In this way we may view $\mu$ as a lift of the measure $\bar \mu$ on $\GL_2(\overline{F})$.
\end{remark}

In \cite{fesenko-loop}, Fesenko points out that the ring of subsets of $F$ generated by sets of the form $\alpha + t_1^i t_2^j O_F$ coincides with the ring generated by $\alpha + t_2^n p^{-1}(S)$ with $S$ measurable subsets of $\overline{F}$, and so the measure defined on $F$ is exactly a lift of the measure on $\overline{F}$. The above remark is the $\GL_2$ version of this statement, i.e. that $\mu$ can be viewed as a lift of the invariant measure on $\GL_2(\overline{F})$.

\begin{remark}
In \cite{kim-lee-spherical} and \cite{lee-iwahori}, Kim and Lee use combinatorial methods to study Hecke Algebras over groups closely related to $\GL_2(F)$ in the absence of an invariant measure. In an unreleased manuscript \cite{kim-lee-measure}, they construct a full $\sigma$-algebra of subsets of $\GL_2(F)$ on which they may define a measure, rather than just a ring of subsets. The measure $\mu_{KL}$ they construct may also be seen to take values in $\mbb{R} (\!(X)\!)$, but in fact their measure is confined to the smaller space of monomials.

The measure $\mu_{KL}$ may be recovered by "taking the dominant term" of the measure $\mu$ we have defined in this chapter, and so we may in fact interpret the two as being "infinitesimally close". The "truncated" convolution product defined in \cite{kim-lee-spherical} and \cite{lee-iwahori} can be expressed in terms of an integral against $\mu_{KH}$, and so it would be interesting to consider whether one can use the integral against $\mu$ defined in the following section to study "non-truncated" convolution product. It has been suggested by van Urk that in order to attack this problem, one must first extend the ring of measurable subsets $\mc{R}$ to include certain ``large" sets which have infinite measure $X^{-j}$ (in accordance with the interpretation of $X$ as an infinitesimal element). In his upcoming paper \cite{wester-fourier} he discusses an improvement of Fesenko's original approach which takes into account higher powers of the second local parameter, and similar constructions may be applicable here.
\end{remark}

\section{Integration on $\GL_2(F)$}  \label{sec:integration}

In classical analysis, the existence of an invariant measure on a locally compact space $X$ is synonymous with the existence of an invariant integral, i.e. a linear functional on the space of continuous, compactly supported, complex-valued functions on $X$.

In this section, we will consider a nice class of functions on $\GL_2(F)$ for which we can construct an integral against the measure defined in the previous section. These functions will be the analogue of the functions at the residue level which are locally constant and compactly supported. They are also analogous to the integrable functions against Fesenko's measure on $F$ as defined in \cite{fesenko-aoas1}.

\begin{definition}
Let $f : \GL_2(F) \rightarrow \mbb{C} (\!(X)\!)$ be a function of the form $$f = \sum_{i=1}^n c_i \mbb{I}_{U_i}$$ where each $c_i \in \mbb{C}(\!(X)\!)$ and $\mbb{I}_{U_i}$ is the indicator function of a dd-set $U_i$. Suppose that the $U_i$ are pairwise disjoint, and let $\mu$ be the measure on $\GL_2(F)$ satisfying $\mu(K)=1$ as in the previous section. We define the integral of $f$ against $\mu$ to be $$\int_{\GL_2(F)} f(g) d \mu (g) = \sum_{i=1}^n c_i \mu(U_i).$$ We also define the integral of a function which is zero outside finitely many points to be $0$.
\end{definition}

\begin{remark}
If $f$ is an integrable function on $\GL_2(F)$ and $E$ is a measurable subset of $\GL_2(F)$, the function $f(g) \mbb{I}_E (g)$ is also integrable, and we may define $$\int_E f(g) d \mu(g) = \int_{\GL_2(F)} f(g) \mbb{I}_E(g) d \mu (g).$$
\end{remark}

\begin{proposition}
Let $R_G$ be the vector space generated by the simple functions $f= \sum_{i=1}^n c_i \mbb{I}_{U_i}$ and the functions which are zero outside of a finite set. Then the map $f(g) \mapsto \int_{\GL_2(F)} f(g) d \mu(g)$ is a well-defined, left-invariant linear functional $R_G \rightarrow \mbb{C}(\!(X)\!)$.
\end{proposition}

\begin{proof}
Suppose that $$f = \sum_i c_i \mbb{I}_{U_i} = \sum_j d_j \mbb{I}_{V_j}$$ are two different ways of expressing $f$ as a simple function. We need to show that $\sum_i c_i \mu(U_i) = \sum_j d_j \mu(V_j).$

By finite additivity of the measure and the property that $\mbb{I}_{A \cup B} = \mbb{I}_A + \mbb{I}_B$ for disjoint sets $A$ and $B$, we may assume that each $U_i$ and each $V_j$ is a dd-set of the form $A \backslash \bigcup_r B_r$ with $A$ and all $B_r$ distinguished sets.

Furthermore, we may arrange that each $U_i$ is in fact a distinguished set as follows. Suppose $U_i = A_i \backslash \bigcup_r B_{i,r}$. For a given $r$, let $W_{i,r} = (B_{i,r} \backslash \bigcup_{k} U_k)$ be the part of $B_{i,r}$ which is not contained in any $U_k$. Since $W_{i,r}$ is disjoint from all $U_k$ and $V_j$, adding $c_i \mbb{I}_{W_{i,r}}$ to both expressions for $f$ changes both integrals by the same value. Doing this for all $r$ and for all $i$ then leaves the left hand expression for $f$ (after simplification) in the form $\sum_i c_i \mbb{I}_{U'_i}$ with $U'_i$ a distinguished set.

For the next reduction, Lemma \ref{intproperty} implies that either $U_1$ is the union $\bigcup_{j \in J} V_{j}$ of some collection of the $V_j$, or the union $\bigcup_{i \in I} U_{i}$ of $U_1$ with some other of the $U_i$ is equal to one of the $V_j$ (which we may say is $V_1$ after relabelling). In the first case we have $f = c_1 \mbb{I}_{U_1} + \sum_{i >1} c_i \mbb{I}_{U_i} = \sum_{j \in J} d_j \mbb{I}_{V_j} + \sum_{j \notin J} d_j \mbb{I}_{V_j}$, and in the second case we have instead $f = \sum_{i \in I} c_i \mbb{I}_{U_i} + \sum_{i \notin I} c_i \mbb{I}_{U_i} = d_1 \mbb{I}_{V_1} + \sum_{j>1} d_j \mbb{I}_{V_j}$. 

By induction on the lengths of the sums (since the result is clear when there is only one set on each side), it then suffices to prove the result for expressions of the form $f^* = c_1 \mbb{I}_{U_1} = \sum_j d_j \mbb{I}_{V_j}$ and $f^\dagger= \sum_i c_i \mbb{I}_{U_i} = d_1 \mbb{I}_{V_1}$. Moreover, since $V_1$ is of the form $A \backslash \bigcup_\ell B_\ell$, and each $B_\ell$ is necessarily disjoint from all of the $U_1$, adding $\sum_\ell d_1 \mbb{I}_{B_\ell}$ to both sides of $f^\dagger$ leaves us in the $f^*$ case.

After dividing by the constant, it remains to prove that if $\mbb{I}_U = \sum d_j \mbb{I}_{V_j}$ with $U$ distinguished and $V_j$ disjoint dd-sets with $U = \bigcup_j V_j$ then we have $\mu(U) = \sum_j d_j \mu(V_j)$. To do so, we first note that since the $V_j$ are disjoint and sum to $U$ we must have $d_j =1$ for all $j$, and the result then follows from finite additivity of the measure.

It remains to show that the integral is a left-invariant linear map. But if $f$ and $g$ can be written as sums of indicator functions then so can $f+g$, so additivity of the integral follows from the definition. Similarly, if $f=\sum c_i \mbb{I}_{U_i}$ and $c \in \mbb{C}(\!(X)\!)$ then $cf = \sum c\cdot c_i \mbb{I}_{U_i}$ and $\int cf \ d \mu = \sum c \cdot c_i \mu(U_i) = c \sum c_i \mu(U_i)= c \int f d \mu$. Left-invariance follows immediately from the corresponding property of the measure.
\end{proof}

\begin{example}
It is expected that the convolution product of two integrable functions will be important, and so as a first example we compute the convolution product of the characteristic function of $K$ with itself.

Let $f(x)=\mbb{I}_{K}(x)$. Then $(f * f)(x) = \int_G f(y)f(y^{-1}x) d \mu (y)$. The integrand is nonzero (and hence equal to $1$) if and only if both $y \in K$ and $y^{-1} x \in K$. The second condition can be rewritten as $x \in yK$, and by the first condition we have $yK=K$, and so the integrand is nonzero if and only if both $y \in K$ and $x \in K$. Thus we have $(f * f)(x) = \int_G f(x)f(y) d \mu (y) = f(x) \mu(K) = f(x)$.
\end{example}

We will now work towards showing that the integral can be factored into a multiple integral over $F$. Since we already know significantly more about integration over $F$ due to \cite{fesenko-aoas1}, \cite{fesenko-loop}, this allows us to deduce some important properties of the integral over $\GL_2(F)$. 

We begin by showing that, for characteristic functions of the $K_{i,j}$, we may take integrals over $F$ in matrix coordinates.

\begin{lemma} \label{coordinate}
Let $f = \mbb{I}_{K_{i,j}}$, and let $g = \bigg( \scalebox{0.8}{$\begin{array}{cc}
\alpha & \beta    \\
\gamma & \delta
 \end{array} $} \bigg) \in \GL_2(F)$. Then
$f \bigg(
\bigg( \scalebox{0.8}{$\begin{array}{cc}
\alpha & \beta    \\
\gamma & \delta
 \end{array} $} \bigg) \bigg)$ is integrable on $F$ with respect to each of $\alpha, \beta, \gamma, \delta$, and the resulting functions remain integrable with respect to each of the remaining variables.
\end{lemma}

\begin{proof}
By the definitions of $f$ and $K_{i,j}$, $f(g)=1$ if and only if the following conditions are satisfied simultaneously: (1) $\alpha \in 1+t_1^i t_2^j O_F$, (2) $\beta \in t_1^i t_2^j O_F$, (3) $\gamma \in t_1^i t_2^j O_F$, (4) $\delta \in 1+ t_1^i t_2^j O_F$; otherwise $f(g)=0$. We can thus express $f(g)$ as a product of functions
\begin{equation*}
f(g)=\mbb{I}_{1+t_1^i t_2^j O_F}(\alpha) \mbb{I}_{t_1^i t_2^j O_F}(\beta) \mbb{I}_{t_1^i t_2^j O_F} (\gamma) \mbb{I}_{1+t_1^i t_2^j O_F}(\delta)
\end{equation*}
with $\alpha, \beta, \gamma, \delta$ appearing independently of each other. If we treat any three of the variables as a constant, we obtain a multiple of an indicator function of a measurable subset of $F$ (which is an integrable function on $F$) in the remaining variable. Integrating over this variable leaves a product of integrable functions of a similar form, but of shorter length, and so by induction one sees that we may successively integrate over each variable. 
\end{proof}

If we want to integrate successively over the entries, it is possible that the result may depend on the order that we integrate out the variables. Indeed, in \cite{morrow-fubini} Morrow shows that Fubini's Theorem may fail in the higher dimensional case under a nonlinear change of variables. Since we will later want to make certain changes of variables which may permute the order of the differentials, it is thus important that we check that the multiple integrals in question do not depend on the order of integration.

\begin{lemma} \label{easyfubini}
Let $f = \mbb{I}_{K_{i,j}}$, and let $g = \bigg( \scalebox{0.8}{$\begin{array}{cc}
\alpha & \beta    \\
\gamma & \delta
 \end{array} $} \bigg) \in \GL_2(F)$. The multiple integral 
\begin{equation*}
 \int_{F} \int_{F} \int_F \int_{F}  f \bigg(
\bigg( \scalebox{0.8}{$\begin{array}{cc}
\alpha & \beta    \\
\gamma & \delta
 \end{array} $} \bigg) \bigg) \ d \alpha \ d \beta \ d \gamma \ d \delta,
\end{equation*}
is well-defined, i.e. it is independent of the order of integration. Here $d \alpha = d \mu_F(\alpha)$, $d \beta = d \mu_F(\beta)$, and so on, where $\mu_F$ denotes the measure on $F$ defined by Fesenko (reviewed here in Section \ref{sec:review}).
\end{lemma}

\begin{proof}
It is enough to show that $$\int_F \int_F \lambda \mbb{I}_A(x) \mbb{I}_B(y) dx dy$$ is well-defined for $A, B$ measurable subsets of $F$ and $\lambda \in \mbb{C}$, since the desired result follows from this by induction. We have
\begin{align*}
\int_F \left( \int_F \lambda \mbb{I}_A(x) \mbb{I}_B(y) dx \right) dy&= \int_F \lambda \mbb{I}_B(y) \left( \int_F \mbb{I}_A(x) dx \right) dy \\
 &= \int_F \lambda \mu_F(A) \mbb{I}_B(y) dy \\
 & = \lambda \mu_F(A) \mu_F(B) \\
 & = \int_F \lambda \mu_F(B) \mbb{I}_A(x) dx \\
 &= \int_F \lambda \mbb{I}_A(x) \left( \int_F \mbb{I}_B(y) dy \right) dx \\
 &= \int_F \left( \int_F \lambda \mbb{I}_A(x) \mbb{I}_B(y) dy \right) dx.
\end{align*}
\end{proof}

\begin{remark}
More generally, for a function $h(x,y)$ on $F \times F$ which is integrable with respect to each of $x$ and $y$ and can be written as a product $h(x,y) = h_1(x) h_2(y)$, the integral $\int_F \int_F h(x,y) dx dy$ is well-defined. As noted above, in classical analysis one has Fubini's Theorem, which ensures that \textit{any} integrable function $h(x,y)$ satisfies the same property, but this is not necessarily true in the higher dimensional case. For the remainder of this text, however, we will only be working with functions where one can "separate the variables", and so the above Lemma will be good enough here.
\end{remark}

We now establish the main result of this section.

\begin{theorem} \label{intfactor}
Let $f \in R_G$, and let $\bigg( \scalebox{0.8}{$\begin{array}{cc}
\alpha & \beta    \\
\gamma & \delta
 \end{array} $} \bigg) \in \GL_2(F)$. Then
\begin{displaymath}
\begin{split} \begin{flalign*} \int_{\GL_2(F)} f(g) \ d \mu(g) = && \end{flalign*} \\
\frac{q^3}{(q^2-1)(q-1)} \int_{F} \int_{F} \int_F \int_{F} \frac{1}{|\alpha \delta - \beta \gamma|_F^2} f \bigg(
\bigg( \scalebox{0.8}{$\begin{array}{cc}
\alpha & \beta    \\
\gamma & \delta
 \end{array} $} \bigg) \bigg) \ d \alpha \ d \beta \ d \gamma \ d \delta.
\end{split}
\end{displaymath}
\end{theorem}

\begin{proof}
First we check that the integrals agree for $f=\mbb{I}_{K_{i,j}}$. In this case, the left hand integral is equal to the volume of $K_{i,j}$. By our choice of normalisation $\mu(K)=1$, we have  $\mu(K_{i,j}) = \dfrac{q^3}{(q^2-1)(q-1)} q^{-4i}X^{4j}$, and so we must check that this agrees with the right hand expression.

By the definition of $K_{i,j}$, we have $$f(g) = \mbb{I}_{1+t_1^it_2^jO_F} (\alpha) \cdot \mbb{I}_{t_1^it_2^jO_F} (\beta) \cdot  \mbb{I}_{t_1^it_2^jO_F} (\gamma) \cdot \mbb{I}_{1+t_1^it_2^jO_F} (\delta).$$ Furthermore, for $g \in K_{i,j}$, $\det g = \alpha \delta - \beta \gamma \in 1 + t_1^i t_2^j O_F$, hence $| \det g|_F^2 = 1$. We can now directly compute the integrals on the right hand side to obtain $\mu_F(1+t_1^it_2^j O_F)^2\mu_F(t_1^it_2^j O_F)^2 = q^{-4i}X^{4j}$. Multiplying by the factor $\dfrac{q^3}{(q^2-1)(q-1)}$ then yields the same value as the left hand side.

Since both integrals are linear, to complete the proof of the Theorem we must show that the integrals agree for $f=\mbb{I}_{hK_{i,j}}$ with $h \in \GL_2(F)$. Since the integral on the left is left-invariant, this is equivalent to proving that the integral on the right is left-invariant.

To show this, it is enough to consider translation by elementary matrices, since these generate $\GL_2(F)$. First suppose $h= \bigg( \scalebox{0.8}{$\begin{array}{cc}
x & 0    \\
0 & 1
 \end{array} $} \bigg)$ with $x \in F^\times$. Then $hg= \bigg( \scalebox{0.8}{$\begin{array}{cc}
x \alpha & x \beta    \\
\gamma & \delta
 \end{array} $} \bigg)$, $|\det (hg)|_F^2 =  |x|_F^2 |\det g|_F^2$, and under the change of variables $g \mapsto hg$ the differentials become $d \alpha \mapsto d (x \alpha) = |x|_F d \alpha$, $d \beta \mapsto d(x \beta) = |x|_F d \beta$, $d \gamma \mapsto d \gamma$, $d \delta \mapsto d \delta$. The factors of $|x|_F$ then cancel with those appearing in $|\det (hg)|_F^2$, leaving the integral invariant under this change of variables. 

The case $h= \bigg( \scalebox{0.8}{$\begin{array}{cc}
1 & 0    \\
0 & x
 \end{array} $} \bigg)$ works almost identically, except the factors of $|x|_F$ come instead from $d \gamma$ and $d \delta$. For $h= \bigg( \scalebox{0.8}{$\begin{array}{cc}
1 & u    \\
0 & 1
 \end{array} $} \bigg)$ with $u \in F$, $\det h = 1$, and the change of variables $g \mapsto hg$ results in an additive shift inside the differentials. Since the measure on $F$ is invariant under such transformations, the integral remains invariant. Exactly the same is true when we take $h$ to be a lower-triangular unipotent matrix. 

The final case we must check is $h= \bigg( \scalebox{0.8}{$\begin{array}{cc}
0 & 1    \\
1 & 0
 \end{array} $} \bigg)$, which has determinant $-1$ (and $|-1|_F=1$) and simply permutes the differentials when we make the change $g \mapsto hg$, which leaves the integral invariant by Lemma \ref{easyfubini}.
\end{proof}

While this Theorem is important on its own, since it gives a way of computing certain integrals much more easily by reducing to integrals over $F$, it also has several important Corollaries.

\begin{corollary}
For $(i,j)>(0,0)$, the unique $\mbb{R}(\!(X)\!)$-valued measure on $\GL_2(F)$ which is compatible with the measure on $F$ (in the sense of the above Theorem) satisfies
\begin{displaymath}
\mu(K_{i,j}) = \frac{q^3}{(q^2-1)(q-1)} q^{-4i}X^{4j}.
\end{displaymath}
\end{corollary}

\begin{corollary} \label{GL2unimodular}
This integral, and hence the measure defined in the previous section, is both left and right translation-invariant. 
\end{corollary}

\begin{proof}
By Theorem \ref{intfactor}, it suffices to check invariance of the integral on the right hand side. However, this can be done in exactly the same way that we checked left-invariance in the above proof. 
\end{proof}

Recall that a series absolutely converges in $\mbb{R}(\!(X)\!)$ if it converges and in addition the coefficient of each $X^j$ is an absolutely convergent series in $\mbb{R}$.

\begin{corollary} \label{countadd}
The measure on $\GL_2(F)$ is countably additive in the following refined sense. If $U= \bigcup_n U_n \in \mc{R}$ such that $\{ U_n \}$ are countably many disjoint measurable sets and $\sum_n \mu(U_n)$ is an absolutely convergent series in $\mbb{R}(\!(X)\!)$ then $\mu(\bigcup_n U_n) = \sum_n \mu(U_n)$. 
\end{corollary}

\begin{proof}
For the measure $\mu_F$ on $F$ this is in \cite{fesenko-loop}, and Theorem \ref{intfactor} allows us to reduce to this case as follows. First of all we may assume each $U_n=C_n \backslash \bigcup_m D_{m,n}$ with $C_n, D_{m,n}$ finitely many distinguished sets by finite additivity. 

By the definition of the integral we have
\begin{equation*}
\mu\left(C_n \backslash \bigcup_m D_{m,n} \right) = \int_{\GL_2(F)} \mbb{I}_{\bigcup_n \left( C_n \backslash \bigcup_m D_{m,n} \right)}(g) dg,
\end{equation*}
and since the union is disjoint the integrand is equal to 
\begin{equation*}
\sum_n \mbb{I}_{C_n \backslash \bigcup_m D_{m,n}} (g).
\end{equation*} 

By applying Theorem \ref{intfactor} in the case $f(g)=\sum_n \mbb{I}_{C_n \backslash \bigcup_m D_{m,n}} (g)$, we obtain the integral
\begin{equation}
\begin{split}
\frac{q^3}{(q^2-1)(q-1)} \int_{F} \int_{F} \int_F \int_{F} \frac{1}{|\alpha \delta - \beta \gamma|_F^2} \sum_n \mbb{I}_{C_n \backslash \bigcup_m D_{m,n}} \bigg(
\bigg( \scalebox{0.8}{$\begin{array}{cc}
\alpha & \beta    \\
\gamma & \delta
 \end{array} $} \bigg) \bigg) \\
 \ d \alpha \ d \beta \ d \gamma \ d \delta.
\end{split}
\end{equation}

Since now we have a (multiple) integral over $F$, we can take the sum outside the integral to obtain
\begin{equation}
\begin{split}
\sum_n \frac{q^3}{(q^2-1)(q-1)} \int_{F} \int_{F} \int_F \int_{F} \frac{1}{|\alpha \delta - \beta \gamma|_F^2}  \mbb{I}_{C_n \backslash \bigcup_m D_{m,n}} \bigg(
\bigg( \scalebox{0.8}{$\begin{array}{cc}
\alpha & \beta    \\
\gamma & \delta
 \end{array} $} \bigg) \bigg) \\
\ d \alpha \ d \beta \ d \gamma \ d \delta,
\end{split}
\end{equation}
which is equal to the previous integral if this sum absolutely converges. However, we can apply Theorem \ref{intfactor} to each of the functions
\begin{equation*}
\mbb{I}_{C_n \backslash \bigcup_m D_{m,n}}(g)
\end{equation*} 
to rewrite this as $$\sum_n \int_{\GL_2(F)} \mbb{I}_{C_n \backslash \bigcup_m D_{m,n}}(g) dg,$$ which equals $$\sum_n \mu \left( C_n \backslash \bigcup_m D_{m,n} \right)$$ by the definition of the integral, and this absolutely converges by assumption. This shows that the two integrals (1) and (2) are equal, and hence we have $$\mu\left( \bigcup_n \left( C_n \backslash \bigcup_m D_{m,n} \right) \right) =\sum_n \mu \left( C_n \backslash \bigcup_m D_{m,n} \right).$$
\end{proof}

\begin{remark} 
We may in fact use the above Corollary to extend the class of measurable sets to all non-ddd-sets which are countable disjoint unions of dd-sets such that the relevant series converges absolutely.
\end{remark}

We end this section with the computation of certain "reasonable" integrals (i.e. not simply characteristic functions of ddd-sets) which do not take values $\mbb{C}$.

\begin{example}
Let $D = \{ g \in M_2(O_F) : \det g \neq 0 \}$ be the "punctured unit disc" in $\GL_2(F)$. For a positive integer $s$, we would like to compute the integral $$I= \int_{D} | \det g|^s_F dg := \int_{\GL_2(F)} |\det g |_F^s \mbb{I}_{D}(g) dg.$$ However, this integral does not converge. 

Indeed, we may first write $D$ as the disjoint union over all $(i,j) \ge (0,0)$ of the "circles" $D_{i,j} = \{ g \in M_2(O_F) : \det g \in t_1^i t_2^j O_F^\times \}$, on which the function $| \det g |_F^s$ takes the constant value $q^{-is}X^{js}$. We may realise the $D_{i,j}$ as cosets of $K = \GL_2(O_F)$ in $D$ by observing that $K$ is the kernel of the map $D \rightarrow \mbb{C}(\!(X)\!)$ sending a matrix $g$ to $|\det g|_F$, and so in particular $\mu(D_{i,j})=1$ for all $i$ and $j$. We thus have $$\int_{D} | \det g|^s_F dg = \sum_{(i,j) \ge (0,0)} q^{-is} X^{js} = \sum_{i=0}^\infty q^{-is} + \sum_{j=1}^\infty \left( \sum_{i=- \infty}^\infty q^{-is} \right) X^{js},$$ and while the first term converges for $s>0$, the inner sum in the second term does not converge for any value of $s$.

However, we may try to approximate this integral by integrating over larger and larger "quarter planes". More precisely, for a pair of integers $m$ and $n$, define $Q_{m,n}$ to be the disjoint union of all $D_{i,j}$ over $i \ge m$ and $j \ge n$ such that $(i,j) \ge (0,0)$. Clearly we have $D = \bigcup_{m \in \mbb{Z}} Q_{m,0}$, or in other words if we integrate over $Q_{m,0}$ for large negative $m$ we should approach the integral $I$. Indeed, if $m<0$ then $$I_{m} = \int_{Q_{m,0}} | \det g |_F^s dg = \sum_{i=0}^\infty q^{-is} + \sum_{j=1}^\infty \sum_{i=m}^\infty q^{-is} X^{js},$$
and now each term converges for $s>0$, giving $$I_m = \frac{1}{1-q^{-s}} + \sum_{j=1}^\infty \frac{q^{-ms}}{1-q^{-s}} X^{js}.$$
\end{example}

\begin{remark}
From the point of view of the "identity" $\sum_{n \in \mbb{Z}} q^n=0$ discussed in \cite{fesenko-aoas1}, the integral $I$ above "agrees" with the classical determinant integral in the one-dimensional case, which has value $(1-q^{-s})^{-1}$. Alternatively, if one considers $X$ as an infinitesimal element, taking the limit as $m \rightarrow - \infty$ of the "dominant term" of $I_m$ also gives $(1-q^{-s})^{-1}$ (recall from the previous section that the measure of Kim and Lee in \cite{kim-lee-measure} may also be seen as taking the "dominant term" of our measure). Note that in the one-dimensional case $s$ may be any complex number with Re$(s) > 0$, whereas we must restrict to positive integers due to the appearence of the indeterminate $X$.
\end{remark}

We conclude with an example of an integral which gives a power series with nonconstant coefficients.

\begin{example}
Let $D_{i,j}$ be as in the previous example, and let $T$ be the disjoint union $T = \bigcup_{i \ge j \ge 0} D_{i,j}$. For an integer $s>0$ we have $$\int_T | \det g|_F^s = \sum_{j=0}^\infty \sum_{i=j}^\infty q^{-is} X^{js} = \sum_{j=0}^\infty \frac{q^{-js}}{1-q^{-s}} X^{js}.$$
\end{example}

\begin{appendix}
\section{Comparing with Morrow's integral} \label{sec:morrow}

The formula in Theorem \ref{intfactor} agrees (up to a constant multiple) with the integral defined by Morrow in \cite{morrow-gln}. In this short section we will briefly compare the two different approaches. First of all, we recall Morrow's method of integration by lifting.

\begin{definition}
Let $g : \overline{F} \rightarrow \mbb{C}$ be a complex-valued function on the residue field of $F$, let $a \in F$, and let $n \in \mbb{Z}$. The lift of $g$ at $(a,n)$ is the function $$g_{a,n}(x) = \begin{cases}
g\left( \overline{(x-a)t_2^{-n}} \right) & \text{if } x \in a + t_2^n \oo_F \\
0 & \text{otherwise}.
\end{cases} $$
\end{definition}

\begin{definition}
Let $\mc{L}(F)$ be the $\mbb{C}(\!(X)\!)$-vector space spanned by functions of the form $x \mapsto g_{a,n}(x) X^r$ for $g$ a Haar-integrable function on $\overline{F}$, $a \in F$, and $n, r \in \mbb{Z}$. Elements of $\mc{L}(F)$ are called integrable functions on $F$.
\end{definition}

\begin{theorem}
There exists a unique linear functional $\int_F ( \cdot ) dx$ on $\mc{L}(F)$ which is invariant under the action of $F$ by additive translations and satisfies $$ \int_F g_{a,n} (x) dx = X^n \int_{\overline{F}} g(u) du$$ for any Haar integrable function $g$ on $F$, where the latter denotes the Haar integral on $\overline{F}$
\end{theorem}

\begin{proof}
See \cite{morrow-2dint}.
\end{proof}

This integral was then extended by Morrow to compute an integral over a finite dimensional $F$-vector space $V$. Since he had shown in \cite{morrow-fubini} that Fubini's Theorem may fail for his integral, he considers the integral of Fubini functions. In other words, he considers only functions $f$ on $V$ for which the iterated integral $$ \int_F \dots \int_F f(x_1, \dots, x_n) dx_1 \dots dx_n$$ (where $n$ is the $F$-dimension of $V$ and we have for simplicity fixed an isomorphism $V \simeq F^n$) does not depend on the order of integration. Morrow then defines the value of $$ \int_V f(x) dx$$ to be the common value of these iterated integrals.

In particular, the space $M_n(F)$ of $n \times n$ matrices with entries in $F$ may be realised in a natural way as an $F$-vector space of dimension $n^2$, and so Morrow's integral naturally gives us a way of integrating on $M_n(F)$.

To define an integral on $\GL_n(F)$, Morrow considers functions $\phi$ on $\GL_n(F)$ such that $\tau \mapsto \phi (\tau) | \det \tau |_F^{-n}$ extends to a Fubini function on $M_n(F)$. He then \textit{defines} the integral of $\phi$ via $$\int_{\GL_n(F)} \phi (\tau) d \tau = \int_{M_n(F)} \phi (x) | \det x |^{-n} dx.$$ In the case $n=2$, upon fixing an isomorphism $M_2(F) \simeq F^4$ the right hand integral can be written in the more familiar form $$ \int_F \int_F \int_F \int_F \phi (x) | \det x |^{-2} dx_1 dx_2 dx_3 dx_4,$$ which we now recognise from the right hand side of Theorem \ref{intfactor}. In fact, this Theorem says exactly that our integral agrees (up to a constant multiple) with the one constructed by Morrow.

The advantage of Morrow's approach is that the integrable functions are related in a direct way to functions on the residue field. This is useful (especially if one considers higher dimensional local fields) if we want to make comparisons between residue levels. This theory is very powerful in the sense that one can deduce many things regarding integrals by appealing to corresponding properties at the residue level which are already known, and it also applies to more general objects than higher local fields. (See \cite{morrow-2dint} for examples of this, including applications to zeta integrals.)

On the other hand, it is too much to hope that \textit{all} of the information about integrals on higher dimensional objects is already contained in the residue field. Indeed, if this were the case then it would not be of much interest to study higher dimensional fields at all. For this reason, it is also useful to have an approach like ours which works directly with the higher dimensional objects, but still satisfies the relevant compatibility conditions with residue structures. This allows us to draw information from the residue level when such information is available, but does restrict us so severely when it is not available.

One of the main advantages our approach has over Morrow's is that the integrals on the right hand side of Theorem \ref{intfactor} are against the measure defined by Fesenko in \cite{fesenko-aoas1}, \cite{fesenko-loop}, and so we obtain Corollary \ref{countadd} as an easy consequence of one of his results, while Morrow does not discuss such "countable additivity" results using his integral.

Another advantage of our more direct approach is that we do not need to write our functions in matrix coordinates in order to integrate them. This leaves our theory more open to modification to work for algebraic groups more general than $\GL_n(F)$, and perhaps even more general spaces.

There is, of course, still significant overlap between Morrow's approach and the approach presented here. In particular, the functions $\mbb{I}_{K_{i,j}}$ are integrable in both cases. This simple observation allows us to transfer some rather interesting consequences of Morrow's work into our setting.

For example, using Proposition 7.1 of \cite{morrow-fubini} and induction, one may construct a function on $\GL_2(F)$ which is integrable in the sense of this text such that the multiple integral in matrix coordinates (see Lemma \ref{easyfubini} above) does not satisfy the Fubini property. All we require to do this is a function of several variables which is integrable in both senses as a starting point, and for this we can just take $\mbb{I}_{K_{i,j}}$ in matrix coordinates.

\end{appendix}

\end{document}